\documentclass[12pt,leqno,a4paper]{article}
\usepackage{amsthm}
\usepackage{indentfirst}
\usepackage{amsmath}
\usepackage{mathrsfs}
\usepackage{paralist}
\usepackage{amssymb,enumerate}
\usepackage{hyperref}
\usepackage{framed}
\usepackage{extarrows}
\usepackage[hmargin=3cm,vmargin=3cm]{geometry}
\usepackage{url}
\usepackage{color}

\swapnumbers
\theoremstyle{theorem}
\newtheorem{thm}{Theorem}[section]
\newtheorem{prop}[thm]{Proposition}
\newtheorem{lem}[thm]{Lemma}

\theoremstyle{definition}
\newtheorem{exmp}[thm]{Example}

\newtheorem{rem}[thm]{Remark}
\newtheorem{thmwot}[thm]{}
\newtheorem{conv}[thm]{Convention}

\numberwithin{equation}{thm}

\newcommand{\Aut}{\operatorname{Aut}}
\newcommand{\Br}{\operatorname{Br}}

\newcommand{\Hom}{\operatorname{Hom}}

\newcommand{\Ind}{\operatorname{Ind}}
\newcommand{\Inn}{\operatorname{Inn}}

\newcommand{\Mod}{\operatorname{Mod}}

\newcommand{\ochar}{\operatorname{char}}
\newcommand{\oo}{\operatorname{o}}
\newcommand{\op}{\operatorname{op}}
\newcommand{\Perf}{\operatorname{Perf}}
\newcommand{\perf}{\operatorname{perf}}

\newcommand{\Res}{\operatorname{Res}}

\newcommand{\cF}{\mathcal{F}}
\newcommand{\cO}{\mathcal{O}}
\newcommand{\cU}{\mathcal{U}}

\newcommand{\grp}[1]{\langle#1\rangle}

\newcommand{\lsup}[1]{{\hspace{1pt}^{#1}\hspace{-0.5pt}}}
\newcommand{\lsub}[1]{{\hspace{1pt}_{#1}\hspace{-0.5pt}}}
\newcommand{\set}[1]{\{\,#1\,\}}
\newcommand{\Set}[1]{\left\{\,#1\,\right\}}
\newcommand{\UHom}{\underline{\Hom}}
\newcommand{\UPerf}{\underline{\Perf}}
\newcommand{\Uperf}{\underline{\perf}}

\title{Stable equivalence of Morita type and quotient fusion system
\footnote{supported by Sichuan Science and Technology Program (No. 2024NSFJQ0070) and by National Natural Science Foundation of China (No. 12271446).}}
\author{Conghui Li\footnote{Conghui Li, School of Mathematics, Southwest Jiaotong University, Chengdu 611756, China, Email: liconghui@swjtu.edu.cn}}
\date{}

\begin{document}

\setlength\abovedisplayskip{1ex plus 0.2ex minus 0.1ex}
\setlength\belowdisplayskip{1ex plus 0.2ex minus 0.1ex}

\maketitle

\begin{abstract}
In this note, we prove that stable equivalences of Morita type between blocks of finite groups induce identification of certain quotient fusion systems under sone assumption.
We also collect some related results for separable equivalences.

\textbf{2020 Mathematics Subject Classification:} 20C20.

\textbf{Keyword:} block of finite group, stable equivalence of Morita type, quotient fusion system, separable equivalence.
\end{abstract}

\section{Introduction}\label{intro}

Throughout this note, all groups are finite; any ring has an identity; modules mean left modules; all (left, right and bi-) modules are assumed to be \emph{finitely generated}; all $\cO$-algebras over a commutative ring $\cO$ are free of finite rank as $\cO$-modules.
The notation for representations of finite groups used in this note are standard; see for example \cite{Linck18a,Linck18b}.
For convenience, we make the following convention.

\begin{conv}\label{conv-O}
Let $p$ be a prime and $\cO$ be a complete discrete valuation ring with the residue field $k$ which is of characteristic $p$ and is large enough for all blocks involved in this note.
\end{conv}

Let $\cO$ be as in Convention \ref{conv-O} and $G$ be a group.
By the completeness assumption of $\cO$, the blocks of $\cO{G}$ correspond bijectively with those of $kG$, which are called blocks of $G$.
We will also identify blocks with their block idempotents by abuse of terminology.

R. Brauer induced defect groups as an important invariant for blocks; see \emph{e.g.} \cite[\S6.1, \S6.2]{Linck18b} for the definition, equivalent characterization and basic properties.
Then J.L. Alperin and M. Brou\'e introduced Brauer pairs as a system of local data for blocks; see \emph{e.g.} \cite[\S6.3]{Linck18b}.
L. Puig formulated the axioms of fusion systems, which can be used to organize these local data into a category; for an exposition of general theory for fusion systems, see \emph{e.g.} \cite{AschKesOliv11}, \cite{Craven11}, \cite[\S\S8.1--8.3]{Linck18b} or \cite{Puig06}.

Puig also introduced source algebras of blocks; see \emph{e.g.} \cite[\S6.4]{Linck18b} for definition and basic properties.
Source algebras are Morita equivalent to block algebras but are better compatible with the local data.
For example, by results of Puig \cite{Puig86}, the fusion systems of blocks can be determined by their source algebras; see \cite[Theorem 8.7.1]{Linck18b} for the statement and a proof.

Morita equivalences, derived equivalences and stable equivalences of Morita type between blocks of finite groups are widely used in the block theory of finite groups.
Morita equivalences obviously imply the derived equivalences and the derived equivalences imply stable equivalences of Morita type by a result of Rickard \cite{Rick91}.
In this note, we will consider stable equivalences of Morita type between blocks of finite groups.
The concept of stable equivalences of Morita type was introduced by Brou\'e \cite{Brou94};  see \cite[\S2.17, \S4.14]{Linck18a} for the definition and basic properties.

An important special class of stable equivalences of Morita type between blocks of finite groups are those induced by bimodules with endopermutation sources.
Such stable equivalences induce an identification of defect groups and fusion systems by a result of Puig \cite[7.6]{Puig99} (see also \cite[Theorem 9.11.2]{Linck18b}).
The main result of this note is that general stable equivalences of Morita type (not necessarily with endopermutation sources) between blocks of finite groups identify certain quotient fusion systems under some assumption.
For more about the notation in the following theorem, see \ref{notation-GH} and \ref{subgrp-direct-prod}.

\begin{thm}\label{mainthm-1}
Let $\cO$ be as in Convention \ref{conv-O} and $G_1,G_2$ be groups.
Assume $b_1$ ($b_2$, resp.) is a block of $\cO{G}_1$ ($\cO{G}_2$, resp.) with defect group $P_1$ ($P_2$, resp.).
Let $M$ be an $(\cO{G_1}b_1,\cO{G_2}b_2)$-bimodule inducing a stable equivalence between $\cO{G_1}b_1$ and $\cO{G_2}b_2$.
Then we may assume that $M$ has a vertex-source pair $(R,V)$ with $R\leq P_1 \times P_2$ and there are source idempotents $i_1 \in (\cO{G_1}b_1)^{P_1}$ and $i_2 \in (\cO{G_2}b_2)^{P_2}$ such that $M$ is isomorphic to a direct summand of
\[ \cO{G_1}i_1 \otimes_{\cO{P_1}} \Ind_R^{P_1\times P_2} V \otimes_{\cO{P_2}} i_2\cO{G_2}, \]
and the following hold for $M,R$ and the fusion system $\cF_1$ ($\cF_2$, resp.) of $\cO{G_1}b_1$ ($\cO{G_2}b_2$, resp.) on $P_1$ ($P_2$, resp.) determined by $i_1$ ($i_2$, resp.).
\begin{compactenum}[(1)]
\item
$M^*$ induces the inverse of the stable equivalence induced by $M$, and $M^*$ has $(R^\sharp,V^*)$ as a vertex-source pair.
\item
The defects of $b_1$ and $b_2$ are equal, \emph{i.e.} $|P_1|=|P_2|$.
\item
The projections $\pi_1\colon R \to P_1$ and $\pi_2\colon R \to P_2$ are both surjective.
\item
Set $R_1 = R\cap P_1$ and $R_2 = R\cap P_2$.
Then $|R_1|=|R_2|$ and there is an isomorphism of groups $\theta\colon P_1/R_1 \to P_2/R_2$ such that
\[ R = \Set{ (x_1,x_2) \in P_1 \times P_2 \mid x_2R_2 = \theta(x_1R_1) }, \]
and $R/(R_1R_2) \cong \Delta_\theta(P_1/R_1)$.
\item
$R_1$ is strongly $\cF_1$-closed if and only if $R_2$ is strongly $\cF_2$-closed.
\item
When the equivalent conditions in (5) hold, the isomorphism $\theta$ in (4) induces an isomorphism of fusion systems from $\cF_1/R_1$ to $\cF_2/R_2$.
\item
When $R_1=1=R_2$, there is an isomorphism of groups $\theta\colon P_1 \cong P_2$ which induces an isomorphism of fusion systems from $\cF_1$ to $\cF_2$.

\end{compactenum}
\end{thm}

We also collect some related results when $\cO$ is of characteristic $0$ concerning separable equivalences, which was first defined by L. Kadison \cite{Kad95} and was used by M. Linckelmann in \cite{Linck11}.
The proof for these results are just slight generalizations of some arguments of Linckelmann in \cite[\S4.13]{Linck18a}

\begin{thm}\label{mainthm-2}
Let $\cO$ be as in Convention \ref{conv-O} and $G_1,G_2$ be groups.
Let $b_1$ ($b_2$, resp.) be a block of $\cO G_1$ ($\cO G_2$, resp.) with defect group $P_1$ ($P_2$, resp.).
\begin{compactenum}[(1)]
\item
Assume $\cO{G}_1b_1$ is separably equivalent to $\cO{G}_2b_2$, then the following hold.
\begin{compactenum}[(1.1)]
\item
$\cO{P}_1$ is separably equivalent to $\cO{P}_2$;
\item
When $\cO$ is of characteristic $0$, then the defects of $b_1$ and $b_2$ are equal, \emph{i.e.} $|P_1|=|P_2|$;
in particular, if $\cO{P}_1$ is separably equivalent to $\cO{P}_2$ and $\cO$ is of characteristic $0$, then $|P_1|=|P_2|$.
\end{compactenum}
\item
When $\cO$ is of characteristic $0$ and $\cO{G}_1b_1$ is stably equivalent of Morita type to $\cO{G}_2b_2$, the following hold with the notation in Theorem \ref{mainthm-1}.
\begin{compactenum}[(2.1)]
\item
$\cO{P}_1$ is separably equivalent to $\cO{P}_2$ via $\Ind_R^{P_1\times P_2}V$ and its dual.
\item
$\cO{R}_1$ is separably equivalent to $\cO{R}_2$ via $V$ and $V^*$.
\end{compactenum}
\end{compactenum}
\end{thm}

The structure of this paper is as follows.

We first give some notation and preliminaries in \S\ref{sect:prelim}.
Puig proved that the fusion systems can be determined by source algebras in \cite{Puig86} (see also \cite[Theorem 8.7.1]{Linck18b}); a key ingredient in the proof of our main result is a corollary of this result (Theorem \ref{thm-quot-fusion-source-alg}).
Another ingredient is a Mackey-like formula of Bouc \cite{Bouc10}; see Theorem \ref{Bouc}.

Then we prove Theorem \ref{mainthm-1} in \S\ref{sect:proof-mainthm-1}.
The parts (1)--(4) follows from some known results and part (7) follows immediately from part (6).
So the bulk of the proof is for part (5) and (6).
In \S\ref{sect:sep-equiv}, we prove Theorem \ref{mainthm-2}.

Finally in \S\ref{exmp-rem}, we include an example in which there are indeed stable equivalences of Morita type with non-trivial $R_1$ and $R_2$ as in \ref{mainthm-1} over $k$, and give some related remarks.

\section{Notation and preliminaries}\label{sect:prelim}

In this section, let $\cO$ be an arbitrary commutative ring unless otherwise stated.

\begin{thmwot}
Let $A,B$ be $\cO$-algebras.
If $U,V$ are $A$-modules, then the notation $U \mid V$ means $U$ is isomorphic to a direct summand of $V$.
As convention, $(A,B)$-bimodules will be identified with $A\otimes_{\cO}B^{\op}$-modules, and in particular right (left) $B$-modules are viewed as left (right) $B^{\op}$-modules.
\end{thmwot}

\begin{thmwot}\label{SEMT}
Let $A,B$ be two $\cO$-algebras.
\begin{compactenum}[(1)]
\item
(\cite{Brou94}) If there are $(A,B)$-bimodule $M$ and $(B,A)$-bimodule $N$ which are projective as left and right modules such that
\[ M \otimes_B N \cong A \oplus U,\quad N \otimes_A M \cong A \oplus V \]
for some projective $A\otimes_{\cO}A^{\op}$-module $U$ and projective $B\otimes_{\cO}B^{\op}$-module $V$, then we say that $A,B$ are stably equivalent of Morita type.
\item
(\cite[Theorem 4.14.2]{Linck18a})
If $A,B$ are indecomposable as $\cO$-algebras, $A$ is not projective as an $A\otimes_{\cO}A^{\op}$-module, $B$ is not projective as a $B\otimes_{\cO}B^{\op}$-module and Krull-Schmidt holds for $\cO$-algebras, then $M,N$ in (1) can be chosen to be indecomposable.
\item
(\cite[Proposition 2.17.15]{Linck18a}) If $A,B$ are symmetric algebras which are stable equivalent of Morita tyhpe via $M$ and $N$, then $A,B$ are stably equivalent of Morita type via $M$ and $M^*$. \hfill $\lrcorner$
\end{compactenum}
\end{thmwot}

\begin{thmwot}
Separable equivalences was first defined by L. Kadison \cite{Kad95} and was used by M. Linckelmann in \cite{Linck11}.
P.A. Bergh and K. Erdmann introduced the related notion of separable division in \cite{BerErd11}.
Many basic properties of separable equivalences are included in \cite{Peac17};
in particular, stable equivalences of Morita type imply separable equivalences.

Let $A,B$ be two $\cO$-algebras.
\begin{compactenum}[(1)]
\item
(\cite{BerErd11}) If there are $(A,B)$-bimodule $M$ and $(B,A)$-bimodule $N$ which are projective as left and right modules such that
\[ M \otimes_B N \cong A \oplus U \]
for some $(A,A)$-bimodule $U$, then we say that $A$ separably divides $B$.
\item
(\cite{Linck11}) If $A$ separably divides $B$ and $B$ separably divides $A$, then we say that $A$ and $B$ are separably equivalent.
Separable equivalence is an equivalent relation between $\cO$-algebras.
\item
(\cite[Proposition 2.12.9]{Linck18a}) If $A,B$ are symmetric algebras which are separably equivalent via $M$ and $N$, then $A,B$ are separably equivalent via $M$ and $M^*$.
\end{compactenum}
Note that for separable equivalences, we can not argument as in \cite[Theorem 4.14.2]{Linck18a} to choose $M,N$ indecomposable even $A,B$ satisfies the condition in (2) of \ref{SEMT}. \hfill $\lrcorner$
\end{thmwot}

\begin{thmwot}\label{notation-GH}
Let $G,H$ be groups.
Then $\cO{G}\otimes_\cO(\cO{H})^{\op}$-bimodules can be identified with $\cO(G\times H)$-modules via the group isomorphism $H \to H^{\op},\ h \mapsto h^{-1}$.
If $G=H$, an $\cO(G \times G)$-module also has an $\cO G$-module structure via the diagonal homomorphism $G \to G \times G,\ g \mapsto (g,g)$.
Let $X$ be a subgroup of $G \times H$ and $V$ be an $\cO X$-module.
For any $(g,h)\in X$ and $v\in V$, we write the action of $(g,h)$ on $v$ as $(g,h).v=g.v.h^{-1}$.
As in \cite[2.1]{Lich25}, there is an isomorphism
\begin{equation}\label{equ-IndXGHV*}
\left(\Ind_X^{G\times H}V\right)^* \cong \Ind_{X^\sharp}^{H\times G}V^*,
\end{equation}
as $(\cO{H},\cO{G})$-bimodules, where $X^\sharp = \Set{ (h,g) \mid (g,h)\in X }$ and $(h,g)\in X^\sharp$ acts on $f\in V^*$ as
\[ [(h,g).f](v) = f(g^{-1}.v.h),\ \forall\, v\in V. \]
\hfill $\lrcorner$
\end{thmwot}

\begin{thmwot}\label{subgrp-direct-prod}
Let $G_1,G_2$ be groups and $X$ be a subgroup of $G_1 \times G_2$.

We denote by $\pi_1$ ($\pi_2$, resp.) the projection from $G_1 \times G_2$ to the first component $G_1$ (the second component $G_2$, resp.), and by abuse of notation, the restrictions of these two projections to $X$ are also denoted as $\pi_1,\pi_2$.
We will often identify $G_1$ ($G_2$, resp.) with $G_1 \times 1$ ($1\times G_2$, resp.) in $G_1 \times G_2$ when there is no confusion.
Set $X_1 = X \cap G_1$, and let $X_1$ also denote the subgroup of $G_1$ such that $X \cap G_1 = X_1 \times 1$ by abuse of notation; similarly for $X_2$.

By \cite[pp.140-141, (4.19)]{Suzuki82}, the subgroup $X$ gives the following data
\[ X_1 \unlhd \pi_1(X) \leq G_1,\quad X_2 \unlhd \pi_2(X) \leq G_2, \]
such that $X_1,X_2$ are both normal subgroup of $X$ and there exist isomorphisms
\[ X/X_1 \cong \pi_2(X),\quad X/X_2 \cong \pi_1(X) \]
and
\[ \theta\colon \pi_1(X)/X_1\cong X/(X_1X_2) \cong \pi_2(X)/X_2; \]
conversely, any system of data as above gives a subgroup of $G_1 \times G_2$.
In particular, we have
\[ X = \Set{ (g_1,g_2) \in \pi_1(X) \times \pi_2(X) \mid g_2X_2 = \theta(g_1X_1) } \]
and $X/[X_1X_2] \cong \Delta_\theta[\pi_1(X)/X_1]$.
(Here, for any group isomorphism $\theta\colon G \to H$, $\Delta_\theta G=\set{(g,\theta(g))\mid g\in G}$ and $\theta$ is abbreviated when $\theta$ is the identity.)
When $X_1=1=X_2$, we have $X=\Delta_\theta[\pi_1(X)]$.
\hfill $\lrcorner$
\end{thmwot}

We will need the following Mackey-like formula for tensor products of bimodules by S. Bouc.

\begin{thm}[Bouc \cite{Bouc10}]\label{Bouc}
Let $G,H,K$ be groups and $X$ ($Y$, resp.) be a subgroup of $G \times H$ ($H \times K$, resp.).
Assume $U$ is an $\cO X$-module and $V$ is an $\cO Y$-module.
Then there is an isomorphism of $(\cO{G},\cO{K})$-bimodules
\[ \left( \Ind_X^{G\times H} U \right) \otimes_{\cO H} \left( \Ind_Y^{H\times K} V \right) \cong \bigoplus_{t\in [\pi_2(X)\backslash{H}/\pi_1(Y)]} \Ind_{X\ast\lsup{(t,1)}Y}^{G\times K} U \otimes_{\cO(X_2\cap\lsup{t}Y_1)} \lsup{(t,1)}V, \]
where
\[ X\ast\lsup{(t,1)}Y = \Set{ (g,k)\in G\times K \mid \exists\, h\in H,\ (g,h)\in X,\, (h^t,k)\in Y } \]
and $(g,k)\in X\ast\lsup{(t,1)}Y$ acts on $u\otimes v \in U \otimes_{\cO(X_2\cap\lsup{t}Y_1)} \lsup{(t,1)}V$ as
\[ (g,k).(u\otimes v) = g.u.h^{-1} \otimes h^t.v.k^{-1} \]
if $h$ is chosen such that $(g,h)\in X$ and $(h^t,k)\in Y$. 
\end{thm}

For convenience, we list in the following lemma some easy observations, which are in fact special cases of the so-called non-injective inductions and restrictions addressed in \cite{LiTian24}.

\begin{lem}\label{lem-res-ind-quotients}
Let $G$ be a group, $N \unlhd G$, $H\leq G$ such that $H$ contains $N$.
Assume $V$ is an $\cO{G}$-module and $U$ is an $\cO{H}$-module.
\begin{compactenum}[(1)]
\item
There is an isomorphism of $\cO(G/N)$-modules
\[ \cO(G/N) \otimes_{\cO{G}} V \cong \cO \otimes_{\cO{N}} V, \]
where on the right side $\bar{g} \in G/N$ acts on $1\otimes v \in \cO \otimes_{\cO{N}} V$ for $v \in V$ as $\bar{g}.(1\otimes v) = 1\otimes g.v$.
\item
$\Res^{G/N}_{H/N} \left[ \cO(G/N) \otimes_{\cO{G}} V \right] \cong \cO(H/N) \otimes_{\cO{H}} \Res^G_HV$.
\item
$\Ind_{H/N}^{G/N} [\cO(H/N)\otimes_{\cO{H}}U] \cong \cO(G/N) \otimes_{\cO{G}} \Ind_H^GU$.
\end{compactenum}
\end{lem}

\begin{proof}
These are special cases of \cite[2.2(2), 3.6, 3.1]{LiTian24}.
But (2) also follows directly from (1) and (3) follows from the association of tensor products:
\[ \Ind_{H/N}^{G/N} [\cO(H/N)\otimes_{\cO{H}}U] = \cO(G/N) \otimes_{\cO(H/N)} \cO(H/N) \otimes_{\cO{H}} U \]
\[ \cong \cO(G/N) \otimes_{\cO{H}} U = \cO(G/N) \otimes_{\cO{G}} \cO{G} \otimes_{\cO{H}} U \cong \cO(G/N) \otimes_{\cO{G}} \Ind_H^GU. \qedhere \]
\end{proof}

The following statement is used in literature (\emph{e.g.} in the proof of \cite[Lemma 9.11.1]{Linck18b}); for convenience, we include a proof of it here.

\begin{lem}\label{lem-inn-P}
Let $\cO$ be as in Convention \ref{conv-O}, $P$ be a $p$-group and $\varphi\in\Aut(P)$.
If $\lsub{\varphi}\cO{P} \cong \cO{P}$ as $(\cO{P},\cO{P})$-bimodules, then $\varphi\in\Inn(P)$.
\end{lem}

\begin{proof}
It follows from \cite[Proposition 2.8.16(i)]{Linck18a} that $\varphi$ induces an inner automorphism of $\cO{P}$ as an $\cO$-algebra.
Thus there is $u \in (\cO{P})^\times$ such that $\varphi(x)=uxu^{-1}$ for any $x\in P$.
Assume $u=\sum_{z\in P}a_zz$ with $a_z\in\cO$.
Then from $u=\varphi(x)^{-1}ux,\, \forall\,x\in P$, we have that
\[ a_{\varphi(x)zx^{-1}} = a_z,\ \forall\, x,z\in P. \tag{$\ast$} \]
There is an action of $P$ on $P$ defined by
\[ x.z=\varphi(x)zx^{-1},\ \forall\, x,z\in P. \]
and any orbit of the above action is of length a power of $p$.
It follows from ($\ast$) that $u$ is an $\cO$-linear combination of orbit sums of the above action.
Since the image of $u$ under the augmentation map $\cO{P} \to \cO{P}/I(\cO{P}) \cong \cO$ is invertible in $\cO$, the above action of $P$ must have a fixed point $y$, which means that
\[ y=x.y=\varphi(x)yx^{-1},\ \forall\, x\in P. \]
Thus $\varphi(x) = yxy^{-1}$.
\end{proof}

\begin{thmwot}\label{quotient-fusion}
For expositions of general theory of fusion system, see \emph{e.g.} \cite[Part I]{AschKesOliv11}, \cite{Craven11} or \cite[\S\S8.1--8.3]{Linck18b}; note that fusion systems in \cite{Linck18b} mean saturated fusion systems as in \cite{AschKesOliv11} and \cite{Craven11}.
Here we just recall some basics of quotient fusion systems following \cite{Craven11}; for the definitions of normal subgroups, strongly closed subgroups and weakly closed subgroups of a fusion system, see \emph{e.g.} \cite[Part I, Definition 4.1]{AschKesOliv11} and see also \cite[Definition 4.55]{Craven11} for an equivalent definition of strongly closed subgroups.

Let $\cF$ be a (not necessarily saturated) fusion system on a $p$-group $P$ and $R\leq P$.

Assume $R$ is strongly $\cF$-closed.
As in \cite[Definition 5.5]{Craven11}, denote by $\bar{\cF}_R$ the image of the natural morphism of fusion systems $\cF \to \cU(P/R)$ with $\cU(P/R)$ the universal fusion system over $P/R$ (see \cite[Definition 5.1]{Craven11} and \cite[Definition 4.2]{Craven11} for the definition of morphisms of fusion systems and the definition of universal fusion systems resp.) and denote by $\grp{\bar{\cF}_R}$ the subcategory of $\cU(P/R)$ generated by $\bar{\cF}_R$.
Then by \cite[Lemma 5.6]{Craven11}, $\grp{\bar{\cF}_R}$ is a fusion system on $P/R$.

Assume $R \unlhd P$ (not normal in $\cF$!).
The quotient system $\cF/R$ is defined as in \cite[Definition 5.9]{Craven11} and $\cF/R$ is a fusion system over $P/R$ by \cite[Proposition 5.10]{Craven11}.

Now we list some results concerning the systems $\grp{\bar{\cF}_R}$ and $\cF/R$ as follows.
\begin{compactenum}[(\ref{quotient-fusion}.1)]
\item\label{quot-fs-item-1}
(\cite[\S5.2]{Craven11})
If $R$ is strongly $\cF$-closed, we have
\[ \cF/R \subseteq \bar{\cF}_R \subseteq \grp{\bar{\cF}_R}, \]
and $\cF/R = \bar{\cF}_R$ implies $\bar{\cF}_R = \grp{\bar{\cF}_R}$, in which case it follows from the definitions of $\cF/R$ and $\bar{\cF}_R$ that:
\emph{for any subgroups $S,T$ of $P$ containing $R$ and any $\psi \in \Hom_\cF(\tilde{S},\tilde{T})$ with $\tilde{S},\tilde{T} \leq P$ such that $R\tilde{S}=S$, $R\tilde{T}=T$, the induced map $\bar{\psi}\colon S/R \to T/R$ in $\cF/R = \bar{\cF}_R = \grp{\bar{\cF}_R}$ is in fact induced by some $\varphi \in \Hom_\cF(S,T)$ stabilizing $R$.}
\item\label{quot-fs-item-2}
(\cite[Proposition 5.13]{Craven11})
If $R \unlhd \cF$ (i.e. $\cF=N_{\cF}(R)$), then $\cF/R = \bar{\cF}_R = \grp{\bar{\cF}_R}$.
\item\label{quot-fs-item-3}
(\cite[Proposition 5.11]{Craven11} or \cite[Theorem 6.2]{Linck07})
If $\cF$ is saturated and $R$ is weakly $\cF$-closed, then $\cF/R$ is saturated.
\item\label{quot-fs-item-4}
(\cite[Theorem 5.14]{Craven11} or \cite{Puig06})
If $\cF$ is saturated and $R$ is strongly $\cF$-closed, then $\cF/R = \bar{\cF}_R = \grp{\bar{\cF}_R}$ is saturated.
\hfill $\lrcorner$
\end{compactenum}
\end{thmwot}

\begin{thmwot}\label{FS-source-alg}
Let $\cO$ be as in Convention \ref{conv-O}, $G$ be a group and $b$ be a block of $G$ with a maximal Brauer pair $(P,e)$, then there is a fusion system $\cF_{(P,e)}(G,b)$ on $P$, called the fusion system of $b$ determined by $(P,e)$; see for example \cite[Definition 8.5.1]{Linck18b}.
With the assumption that $k$ is large enough, $\cF_{(P,e)}(G,b)$ is saturated; see for example \cite[Theorem 8.5.2]{Linck18b}.
Since different maximal Brauer pairs associated with the block $b$ are $G$-conjugate, the fusion systems given by maximal Brauer pairs are isomorphic.
A particular maximal Brauer pair and the fusion system associated to this maximal Brauer pair can be determined by a source idempotent $i\in(\cO{G}b)^P$; see for example \cite[\S8.7]{Linck18b}.
Furthermore, the fusion system determined by the source idempotent $i$ is determined by the $(\cO{P},\cO{P})$-bimodule structure of the source algebra $i\cO{G}i$ by the local fusion of Puig \cite{Puig86}; see also \cite[Theorem 8.7.1]{Linck18b}.
We include here a corollary of Puig's local fusion (we use the statements in \cite[Theorem 8.7.1]{Linck18b}) to quotients of fusion systems given by blocks.
Note that \emph{the assertion in \cite[Theorem 8.7.1(ii)]{Linck18b} still holds when the condition ``$R$ is fully $\cF$-centralised'' is replaced by ``$Q$ is fully $\cF$-centralised''} with some obvious modification in the proof of \cite[Theorem 8.7.1(ii)]{Linck18b}
\footnote{Keep the notation in \cite[Theorem 8.7.1]{Linck18b} and assume $\varphi\colon Q \to R$ is an isomorphism in $\cF$, then it can be proved that $\lsub{\varphi}\cO{R} \mid A$ as $(\cO{Q},\cO{R})$-bimodules when $Q$ is fully $\cF$-centralized.
In fact, by assumption, there is $x\in G$ such that $\varphi(u)=\lsup{x}u$ for any $u\in Q$ and $\lsup{x}e_Q=e_R$.
Let $\nu$ be a local point of $R$ on $\cO{G}b$ such that $A\cap\nu\neq0$ and set $\mu=\lsup{x^{-1}}\nu$.
Since $A\cap\nu\neq0$, we have $\Br_R^A(\nu)e_R\neq0$ and conjugating by $x^{-1}$ implies that $\Br_Q^A(\mu)e_Q\neq0$.
Since $Q$ is fully $\cF$-centralized, it follows from \cite[Proposition 8.7.3(ii)]{Linck18b} that $A\cap\mu\neq0$.
Then the rest of the proof is exactly the same as for \cite[Theorem 8.7.1(ii)]{Linck18b}.}.
\end{thmwot}

\begin{thm}\label{thm-quot-fusion-source-alg}
Let $\cO$ be as in Convention \ref{conv-O}, $G$ be a group and $b$ be a block of $G$ with a defect group $P$ and a source idempotent $i\in(\cO{G}b)^P$.
Denote by $\cF$ the fusion system of the block $\cO{G}b$ on $P$ determined by $i$.
Assume $R\unlhd P$.
Set $\bar{P}=P/R$ and use bar convention for $\bar{P}$.
\begin{compactenum}[(1)]
\item
Assume $R$ is strongly $\cF$-closed.
Then for any two subgroups $S,T$ of $P$ containing $R$, any indecomposable direct summand of $\cO\bar{P} \otimes_{\cO{P}} i\cO{G}i \otimes_{\cO{P}} \cO\bar{P}$ as an $(\cO\bar{S},\cO\bar{T})$-bimodule is isomorphic to $\cO\bar{S} \otimes_{\cO\bar{Q}} \lsub{\bar{\varphi}}\cO\bar{T}$, where $R \leq Q \leq S$ and $\bar{\varphi}\colon \bar{Q} \to \bar{T}$ in $\cF/R$.
\item
If $S,T$ are subgroups of $P$ containing $R$ such that $S$ or $T$ is fully $\cF$-centralized and $\bar{\varphi}\colon \bar{S} \to \bar{T}$ is an $\cF/R$-isomorphism, then $\lsub{\bar{\varphi}}\cO\bar{T}$ is isomorphic to an indecomposable direct summand of $\cO\bar{P} \otimes_{\cO{P}} i\cO{G}i \otimes_{\cO{P}} \cO\bar{P}$ as an $(\cO\bar{S},\cO\bar{T})$-bimodule.
\item
Assume $R$ is strongly $\cF$-closed.
If $T$ is fully $\cF$-centralized containing $R$ and $\varphi\in\Aut(T)$, then $\bar{\varphi} \in \cF/R$ if and only if $\lsub{\bar{\varphi}}\cO\bar{T}$ is an indecomposable direct summand of $\cO\bar{P} \otimes_{\cO{P}} i\cO{G}i \otimes_{\cO{P}} \cO\bar{P}$ as an $(\cO\bar{T},\cO\bar{T})$-bimodule.
\item
If $R$ is weakly $\cF$-closed, then the fusion system $\cF/R$ is generated by $\Aut_{\cF/R}(\bar{T})$ with $T$ running over all fully $\cF$-centralized subgroups of $P$ containing $R$.
\item
If $R$ is strongly $\cF$-closed, then $\cF/R$ is determined by the $(\cO\bar{P},\cO\bar{P})$-bimodule structure of $\cO\bar{P} \otimes_{\cO{P}} i\cO{G}i \otimes_{\cO{P}} \cO\bar{P}$.
\end{compactenum}
\end{thm}

\begin{proof}
Note that it follows from Lemma \ref{lem-res-ind-quotients}\,(2) that
\[ \tag{$\ast$} \Res^{\bar{P}\times\bar{P}}_{\bar{S}\times\bar{T}} \left[ \cO\bar{P} \otimes_{\cO{P}} i\cO{G}i \otimes_{\cO{P}} \cO\bar{P} \right] \cong \cO\bar{S} \otimes_{\cO{S}} i\cO{G}i \otimes_{\cO{T}} \cO\bar{T}. \]

(1) By \cite[Theorem 8.7.1(i)]{Linck18b}, any indecomposable direct summand of $i\cO{G}i$ as $(\cO{S},\cO{T})$-bimodule is isomorphic to $\cO{S} \otimes_{\cO{Q_0}} \lsub{\varphi}\cO{T}$ for some $Q_0 \leq S$ and some $\varphi\colon Q_0 \to T$ in $\cF$.
By ($\ast$), the bimodule $\cO{S} \otimes_{\cO{Q_0}} \lsub{\varphi}\cO{T}$ gives a direct summand $\cO\bar{S} \otimes_{\cO{Q_0}} \lsub{\varphi}\cO\bar{T}$ of $\cO\bar{P} \otimes_{\cO{P}} i\cO{G}i \otimes_{\cO{P}} \cO\bar{P}$ as an $(\cO\bar{S},\cO\bar{T})$-bimodule.
Note that there is an isomorphism of $(\cO\bar{S},\cO\bar{T})$-bimodules
\[ \cO\bar{S} \otimes_{\cO{Q_0}} \lsub{\varphi}\cO\bar{T} \cong \cO\bar{S} \otimes_{\cO\bar{Q}} \lsub{\bar{\varphi}}\cO(\bar{T}), \]
where $Q=Q_0R$ and $\bar{\varphi}\colon \bar{Q} \to \bar{T}$ is the map induced by $\varphi$.
By the assumption that  $R$ is strongly $\cF$-closed, it follows from (\ref{quotient-fusion}.\ref{quot-fs-item-4}) that $\cF/R = \bar{\cF}_R = \grp{\bar{\cF}_R}$ and the description in (\ref{quotient-fusion}.\ref{quot-fs-item-1}) for morphisms in the quotient system applies.
Thus $\bar{\varphi} \in \cF/R$.
By Green's indecomposable theorem, $\cO\bar{S} \otimes_{\cO\bar{Q}} \lsub{\bar{\varphi}}\cO(\bar{T})$ is indecomposable as an $(\cO\bar{S},\cO\bar{T})$-bimodule.
Thus all indecomposable direct summands of $\cO\bar{P} \otimes_{\cO{P}} i\cO{G}i \otimes_{\cO{P}} \cO\bar{P}$ as an $(\cO\bar{S},\cO\bar{T})$-bimodule are of the above form.
This shows (1).

(2) By the definition of $\cF/R$, $\bar{\varphi}$ is induced by certain $\varphi\colon S \to T$ in $\cF$ stabilizing $R$.
Since $S$ or $T$ is fully $\cF$-centralized, it follows from \cite[Theorem 8.7.1(ii)]{Linck18b} that $\lsub{\varphi}\cO{T}$ is an indecomposable direct summand of $i\cO{G}i$ as an $(\cO{S},\cO{T})$-bimodule (see the explanation in \ref{FS-source-alg}).
So
\[ \lsub{\bar{\varphi}}\cO\bar{T} \cong \cO\bar{S} \otimes_{\cO{S}} \lsub{\varphi}\cO{T} \otimes_{\cO{T}} \cO\bar{T} \]
is an indecomposable direct summand of $\cO\bar{S} \otimes_{\cO{S}} i\cO{G}i \otimes_{\cO{T}} \cO\bar{T}$, and then (2) follows from ($\ast$).

(3) The ``only if'' part follows immediately from (2).
For the ``if'' part, assume $\lsub{\bar{\varphi}}\cO\bar{T}$ is an indecomposable direct summand of $\cO\bar{P} \otimes_{\cO{P}} i\cO{G}i \otimes_{\cO{P}} \cO\bar{P}$ as an $(\cO\bar{T},\cO\bar{T})$-bimodule.
Then by (1), there is some $\bar{\psi} \in \Aut_{\cF/R}(\bar{T})$ such that $\lsub{\bar{\varphi}}\cO\bar{T} \cong \lsub{\bar{\psi}}\cO\bar{T}$.
By Lemma \ref{lem-inn-P}, $\bar{\varphi}$ differs from $\bar{\psi}$ by an inner automorphism of $\bar{T}$, and thus $\bar{\varphi} \in \Aut_{\cF/R}(\bar{T})$.

By Alperin's fusion theorem (see e.g. \cite[Theorem 8.2.8]{Linck18b}), $\cF$ is generated by $\Aut_{\cF}(T)$ with $T$ running over all fully $\cF$-centralized subgroups of $P$.
So (4) follows from the definition of $\cF/R$.
Finally, (5) follows from (3) and (4).
\end{proof}

\begin{rem}
Part (4) of Theorem \ref{thm-quot-fusion-source-alg} holds for any saturated fusion system $\cF$ (not necessarily coming from blocks) and any weakly $\cF$-closed subgroup $R$.
\end{rem}

\section{Proof of Theorem \ref{mainthm-1}}\label{sect:proof-mainthm-1}

Let $\cO$ be as in Convention \ref{conv-O} in this section.

\begin{thmwot}\label{setting-SEMT}
We begin with some basic settings and results for stable equivalences of Morita type between blocks of finite groups as in \cite[3.1]{Lich25}, which generalize slightly some arguments in the proof of \cite[Theorem 9.11.2]{Linck18b}.
Let $G_1$\,($G_2$, resp.) be a group and $b_1$\,($b_2$, resp.) be a block of $G_1$\,($G_2$, resp.) with a defect group $P_1$\,($P_2$, resp.).
Assume $\cO{G}_1b_1$ and $\cO{G}_2b_2$ are stably equivalent of Morita type.
We may assume that \emph{both $b_1$ and $b_2$ are not of defect zero}.

Since block algebras are indecomposable symmetric algebras, we can assume a stable equivalence of Morita type between $\cO{G}_1b_1$ and $\cO{G}_2b_2$ is induced by an indecomposable $(\cO{G_1}b_1,\cO{G_2}b_2)$-bimodule $M$ which is projective as left and right module and its dual $M^*$ (see \ref{SEMT}).
So
\[ M \otimes_{\cO{G}_2b_2} M^* \cong \cO{G}_1b_1 \oplus U_1, \quad M^* \otimes_{\cO{G}_1b_1} M \cong \cO{G}_2b_2 \oplus U_2 \]
for some projective $(\cO{G}_1b_1,\cO{G}_1b_1)$-bimodule $U_1$ and some projective $(\cO{G}_2b_2,\cO{G}_2b_2)$-bimodule $U_2$.
As an $\cO(G_1\times G_2)$-module, $M$ belongs to the block $b_1\otimes b_2^{\oo}$ of $G_1 \times G_2$ with $P_1 \times P_2$ as a defect group, where $b_2^{\oo}$ is the block of $\cO{G_2}$ corresponding to the block $b_2$ of $\cO{G_2}^{\op}$ induced by the group isomorphism $G_2 \to G_2^{\op},\ g\mapsto g^{-1}$; see for example \cite[Proposition 8.7.7]{Linck18b}.
Thus we may choose a vertex $R$ of $M$ contained in $P_1 \times P_2$.
As in \cite[3.1]{Lich25}, we can take an $\cO{R}$-source $V$ of $M$ such that
\begin{equation}
M \mid \cO{G_1}i_1 \otimes_{\cO{P_1}} \Ind_R^{P_1 \times P_2} V \otimes_{\cO{P_2}} i_2\cO{G_2},
\end{equation}
where $i_1$\,($i_2$, resp.) is a source idempotent in $(\cO{G_1}b_1)^{P_1}$\,($(\cO{G_2}b_2)^{P_2}$, resp.).
It is readily to see that $M^*$ has $(R^\sharp,V^*)$ as a vertex-source pair.

The main results of \cite[3.1, Remark 3.4]{Lich25} are listed as follows.

\begin{compactenum}[(\ref{setting-SEMT}.1)]\setcounter{enumi}{1}
\item
The $(\cO{G_1}b_1,\cO{P_2})$-bimodule $Mi_2$ can be decomposed as
\[ Mi_2 = N \oplus N', \]
where $N$ is an indecomposable non-projective $(\cO{G_1}b_1,\cO{P_2})$-bimodule and $N'$ is a projective $(\cO{G_1}b_1,\cO{P_2})$-bimodule
\footnote{In \cite[3.1]{Lich25}, the sentence from line -2 on page 247 to line 2 on page 248 states that ``Since $M^* \otimes_{\cO{H}c} M \cong \cO{G}b \oplus L$ for some projective $(\cO{G}b,\cO{G}b)$-bimodule $L$ and $M \otimes_{\cO{G}b} M^* \cong \cO{H}c \oplus L'$ for some projective $(\cO{H}c,\cO{H}c)$-bimodule $L'$, $M\otimes_{\cO{G}b}-$ also induces an equivalence as a functor from the $\cO$-stable category of $(\cO{G}b,\cO{P})$-bimodules to the $\cO$-stable category of $(\cO{H}c,\cO{Q})$-bimodules.''
Here, ``$\cO$-stable category of $(\cO{G}b,\cO{P})$-bimodules'' should be $\UPerf(\cO{G}b,\cO{P})$ or $\Uperf(\cO{G}b,\cO{P})$ and ``$\cO$-stable category of $(\cO{H}c,\cO{Q})$-bimodules'' should be $\UPerf(\cO{H}c,\cO{P})$ or $\Uperf(\cO{H}c,\cO{P})$.}.
\item
We have as $(\cO{G_1}b_1,\cO{P_2})$-bimodules that
\[ N \mid \cO{G_1}i_1 \otimes_{\cO{P_1}} \Ind_R^{P_1 \times P_2} V. \]
\item \label{item-N*0-N0-mid}
We have as $(\cO{P_2},\cO{P_2})$-bimodules that
\[ N^* \otimes_{\cO{G_1}b_1} N \mid \Ind_{R^\sharp}^{P_2\times P_1} V^* \otimes_{\cO{P_1}} i_1\cO{G_1}i_1 \otimes_{\cO{P_1}} \Ind_R^{P_1\times P_2} V, \]
where the isomorphism (\ref{equ-IndXGHV*}) is used.
\item \label{item-iOGi-N*0-N0}
We have as $(\cO{P_2},\cO{P_2})$-bimodules that
\[ i_2\cO{G_2}i_2 \oplus X \cong i_2M^* \otimes_{\cO{G_1}b_1} Mi_2 \cong (N^* \otimes_{\cO{G_1}b_1} N) \oplus Y, \]
where $X$ and $Y$ are both projective as $\cO(P_2\times P_2)$-modules.
\item\label{N^*0-N0-P-P-stable-basis}
$i_2M^* \otimes_{\cO{G_1}b_1} Mi_2$ and $N^* \otimes_{\cO{G_1}b_1} N$ have $P_2 \times P_2$-stable basis;
\item\label{item-OP-N*0-N0}
The non-projective indecomposable direct summands of $i_2\cO{G_2}i_2$ and of $N^* \otimes_{\cO{G_1}b_1} N$ as $\cO(P_2\times P_2)$-modules are the same up to isomorphism; in particular, $\cO{P_2} \mid N^* \otimes_{\cO{G_1}b_1} N$.
\item
(\cite[Remark 3.4]{Lich25})
$\Ind_R^{P_1\times P_2}V \mid i_1N$ and thus $\Ind_R^{P_1\times P_2}V \mid i_1Mi_2$ as $\cO(P_1\times P_2)$-modules.
\end{compactenum}
Finally, we denote by $\cF_1$ ($\cF_2$, resp.) the fusion system of $\cO{G_1}b_1$ ($\cO{G_2}b_2$, resp.) on $P_1$ ($P_2$, resp.) determined by $i_1$ ($i_2$, resp.). \hfill $\lrcorner$
\end{thmwot}

\begin{prop}\label{prop-first}
Keep the notation in Theorem \ref{mainthm-1} and \ref{setting-SEMT}.
\begin{compactenum}[(1)]
\item
The defects of $b_1$ and $b_2$ are equal, \emph{i.e.} $|P_1|=|P_2|$.
\item
The projections $\pi_1\colon R \to P_1$ and $\pi_2\colon R \to P_2$ are both surjective.
\item
There is an isomorphism of groups $\theta\colon P_1/R_1 \to P_2/R_2$ such that
\[ R = \Set{ (x_1,x_2) \in P_1 \times P_2 \mid x_2R_2 = \theta(x_1R_1) }, \]
and $R/(R_1R_2) \cong \Delta_\theta[P_1/R_1]$.
\end{compactenum}
\end{prop}

\begin{proof}
For (1), note first that stable equivalence of Morita type between $\cO{G}_1b_1$ and $\cO{G}_2b_2$ induces an stable equivalence of Morita type between $kG_1\bar{b}_1$ and $kG_2\bar{b}_2$.
Then by the arguments in \cite[\S5]{Xi08} (see also \cite[Proposition 5.7.4]{Zimmer14} or \cite[Proposition 4.14.13]{Linck18a}.), there is an isomorphism between the stable Grothendieck groups $G_0(kG_1\bar{b}_1)/\Pr(kG_1\bar{b}_1)$ and $G_0(kG_2\bar{b}_2)/\Pr(kG_2\bar{b}_2)$.
Thus the Cartan matrices of $kG_1\bar{b}_1$ and $kG_2\bar{b}_2$ have the same elementary divisors.
Since the unique greatest elementary divisor of the Cartan matrix of any block of finite groups is just the order of the defect groups (see \emph{e.g.} \cite[Chapter 3, Theorem 6.35]{NT}), then we have that $|P_1|=|P_2|$.

(2) is contained in \cite[Theorem 6.9]{Puig99} and an alternative proof is given in \cite[Lemma 3.2]{Lich25} (note that the proof there does not need the assumption that $\ochar\cO=0$; in fact this assumption is only needed in \cite[3.3]{Lich25}).
But the proof in \cite[Lemma 3.2]{Lich25} can be even further simplified as follows.
It follows from (\ref{setting-SEMT}.\ref{item-N*0-N0-mid}) and (\ref{setting-SEMT}.\ref{item-OP-N*0-N0}) that as $\cO(P_2\times P_2)$-modules
\[ \cO{P}_2 \mid \Ind_{R^\sharp}^{P_2\times P_1} V^* \otimes_{\cO{P_1}} A_1 \otimes_{\cO{P_1}} \Ind_R^{P_1\times P_2} V, \quad A_1:=i_1\cO{G}_1i_1. \]
Then using Bouc's formula \ref{Bouc} twice shows that
\begin{align*}
& \Ind_{R^\sharp}^{P_2\times P_1} V^* \otimes_{\cO{P_1}} A_1 \otimes_{\cO{P_1}} \Ind_R^{P_1\times P_2} V \\
\cong\ &\Ind_{R^\sharp}^{P_2\times P_1} V^* \otimes_{\cO{P_1}} \Ind_{P_1\times P_1}^{P_1\times P_1}A_1 \otimes_{\cO{P_1}} \Ind_R^{P_1\times P_2} V \\
\cong\ & \Ind_{\pi_2(R)\times P_1}^{P_2\times P_1} (V^* \otimes_{\cO{R_1}} A_1) \otimes_{\cO{P_1}} \Ind_R^{P_1\times P_2} V \\
\cong & \Ind_{\pi_2(R)\times \pi_2(R)}^{P_2\times P_2} (V^* \otimes_{\cO{R}_1} A_1 \otimes_{\cO{R}_1}V),
\end{align*}
where we use $R^\sharp\ast(P_1\times P_1)=\pi_2(R)\times P_1$ and $(\pi_2(R)\times P_1)\ast(P_1\times P_2)=\pi_2(R)\times \pi_2(R)$.
Since as an $\cO(P_2\times P_2)$-module, $\cO{P}_2$ has $\Delta(P_2)$ as a vertex, so we must have that $\pi_2(R)=P_2$; symmetric arguments shows that $\pi_1(R)=P_1$.

Finally, (3) follows from the description of subgroups of direct products of groups in \ref{subgrp-direct-prod}.
\end{proof}

\begin{prop}\label{prop-strongly-closed}
Keep the notation in Theorem \ref{mainthm-1} and \ref{setting-SEMT}.
Then $R_1$ is strongly $\cF_1$-closed if and only if $R_2$ is strongly $\cF_2$-closed.
\end{prop}

\begin{proof}
We prove that the assertion that $R_1$ is strongly $\cF_1$-closed implies that $R_2$ is strongly $\cF_2$-closed; the converse is similar.
Thus we have to show that $\varphi_2(S_2)\leq R_2$ for any $S_2\leq R_2$ and any $\varphi_2\colon S_2 \to P_2$ in $\cF_2$.
By Alperin's fusion theorem (see e.g. \cite[Theorem 8.2.8]{Linck18b} and note that ``fully normalized'' implies ``fully centralized'') and an easy induction, we may assume that there is a fully $\cF_2$-centralized subgroup $T_2$ of $P_2$ containing both $S_2$ and $\varphi_2(S_2)$ such that $\varphi_2$ can be extended to a morphism in $\Aut_{\cF_2}(T_2)$; for convenience, we denote such an extension still by $\varphi_2$.

By \cite[Theorem 8.7.1(iii)]{Linck18b}, $\lsub{\varphi_2}\cO T_2$ is a direct summand of $i_2\cO G_2i_2$ as an $(\cO T_2,\cO T_2)$-bimodule.
Then by (\ref{setting-SEMT}.\ref{item-N*0-N0-mid}) and (\ref{setting-SEMT}.\ref{item-iOGi-N*0-N0}), we have as $(\cO{T_2},\cO{T_2})$-bimodules that
\[ \lsub{\varphi_2}\cO T_2 \mid \Ind_{R^\sharp}^{P_2\times P_1} V^* \otimes_{\cO{P_1}} i_1\cO{G_1}i_1 \otimes_{\cO{P_1}} \Ind_R^{P_1\times P_2} V. \]
By restriction, we have as $(\cO{S_2},\cO{\varphi_2(S_2)})$-bimodules that
\[ \lsub{\varphi_2}\cO \varphi_2(S_2) \mid \Ind_{R^\sharp}^{P_2\times P_1} V^* \otimes_{\cO{P_1}} i_1\cO{G_1}i_1 \otimes_{\cO{P_1}} \Ind_R^{P_1\times P_2} V. \]
Then by \cite[Theorem 8.7.1(i)]{Linck18b}, there exist $Q_1\leq P_1$ and $\varphi_1\colon Q_1 \to P_1 \in \cF_1$ such that
\[ \lsub{\varphi_2}\cO \varphi_2(S_2) \mid \Ind_{R^\sharp}^{P_2\times P_1} V^* \otimes_{\cO{P_1}} \Ind_{\Delta_{\varphi_1}Q_1}^{P_1 \times P_1} \cO \otimes_{\cO{P_1}} \Ind_R^{P_1\times P_2} V. \tag{$\ast$}\]

It follows from the McKay formula, the surjection of the projection of $R$ to the first component and the structure of $R$ from Proposition \ref{prop-first} that
\[ \Res^{P_2\times P_1}_{S_2\times P_1}\Ind_{R^\sharp}^{P_2\times P_1}V^* \cong \Ind_{(S_2\times P_1)\cap R^\sharp}^{S_2\times P_1} V^* = \Ind_{S_2\times R_2}^{S_2\times P_1} V^* \]
and
\[ \Res^{P_1\times P_2}_{P_1\times\varphi_2(S_2)} \Ind_R^{P_1\times P_2} V \cong \Ind_{[P_1\times\varphi_2(S_2)]\cap R}^{P_1\times\varphi_2(S_2)} V. \]
Thus by Bouc's formula (Theorem \ref{Bouc}), $\Res^{P_2\times P_1}_{S_2\times P_1}\left( \Ind_{R^\sharp}^{P_2\times P_1} V^* \otimes_{\cO{P_1}} \Ind_{\Delta_{\varphi_1}Q_1}^{P_1 \times P_1} \cO \right)$ is a direct sum of $(\cO S_2, \cO P_1)$-bimodules of the form
\[ \Ind_{(S_2\times R_1)*\lsup{(t_1,1)}\Delta_{\varphi_1}Q_1}^{S_2\times P_1} V^* \]
for some $t_1\in P_1$.
By the definition of the operation $*$ in Theorem \ref{Bouc}, we have that
\[ (S_2\times R_1)*\lsup{(t_1,1)}\Delta_{\varphi_1}Q_1 = S_2 \times \varphi_1(Q_1\cap R_1). \]
Again by Bouc's formula, $\Ind_{S_2 \times \varphi_1(Q_1\cap R_1)}^{S_2\times P_1} V^* \otimes_{\cO P_1} \Ind_{[P_1\times\varphi_2(S_2)]\cap R}^{P_1\times\varphi_2(S_2)} V$ is a direct sum of $(\cO S_2, \cO\varphi_2(S_2))$-bimodules of the form
\[ \Ind_{[S_2 \times \varphi_1(Q_1\cap R_1)] * \lsup{(u_1,1)}\left\{[P_1\times\varphi_2(S_2)]\cap R\right\}}^{S_2\times\varphi_2(S_2)} V^* \otimes_{\varphi_1(R_1\cap Q_1)\cap R_1} \lsup{(u_1,1)}V \]
for some $u_1\in P_1$.
Again by the definition of the operation $*$ in Theorem \ref{Bouc}, we have that
\[ [S_2 \times \varphi_1(Q_1\cap R_1)] * \lsup{(u_1,1)}\left\{[P_1\times\varphi_2(S_2)]\cap R\right\} = S_2 \times \left[ \varphi_2(S_2) \cap \tau_2^{-1}(\theta(\overline{\varphi_1(R_1\cap Q_1)^{u_1}})) \right], \]
where $\tau_2\colon P_2 \to P_2/R_2$ is the natural surjective map and we use bar convention for $\bar{P}_1 = P_1/R_1$ and $\bar{P}_2=P_2/R_2$.

By ($\ast$) and the above paragraph, as an $(\cO S_2,\cO\varphi_2(S_2))$-bimodule, $\lsub{\varphi_2}\cO \varphi_2(S_2)$ is a direct summand of
\[ \Ind_{S_2 \times \left[ \varphi_2(S_2) \cap \tau_2^{-1}(\theta(\overline{\varphi_1(R_1\cap Q_1)^{u_1}})) \right]}^{S_2\times\varphi_2(S_2)} V^* \otimes_{\varphi_1(R_1\cap Q_1)\cap R} \lsup{(u_1,1)}V. \]
Since $\Delta_{\varphi_2}S_2$ is a vertex of $\lsub{\varphi_2}\cO \varphi_2(S_2)$, $\Delta_{\varphi_2}S_2$ is $S_2\times\varphi_2(S_2)$-conjugate to a subgroup of $S_2 \times \left[ \varphi_2(S_2) \cap \tau_2^{-1}(\theta(\overline{\varphi_1(R_1\cap Q_1)^{u_1}})) \right]$.
Thus we must have
\[ \tau_2(\varphi_2(S_2)) \leq \theta(\overline{\varphi_1(R_1\cap Q_1)^{u_1}}). \]
Now, since $R_1$ is strongly $\cF_1$-closed in $P_1$, $\varphi_1(R_1\cap Q_1) \leq R_1$, thus $\theta(\overline{\varphi_1(R_1\cap Q_1)^{u_1}})=\bar{1}$ and so $\varphi_2(S_2) \leq R_2$.
\end{proof}

\begin{thmwot}
\begin{proof}[Proof of Theorem \ref{mainthm-1}]
It remains to prove part (6) of Theorem \ref{mainthm-1} since part (7) is a special case of part (6).

We will use the bar convention simutaneously for both $\bar{P}_1 = P_1/R_1$ and $\bar{P}_2 = P_2/R_2$, and set $\bar{\cF}_1=\cF_1/R_1$ and $\bar{\cF}_2=\cF_2/R_2$.
We will show that $\theta^{-1}\circ\bar{\varphi}\circ\theta \in \bar{\cF}_1$ for any fully $\cF_2$-centralized subgroup $T_2$ of $P_2$ containing $R_2$ and $\bar{\varphi} \in \Aut_{\bar{\cF}_2}(\bar{T}_2)$.

By Theorem \ref{thm-quot-fusion-source-alg}(2) and Lemma \ref{lem-res-ind-quotients}(2), we have as $(\cO\bar{T}_2,\cO\bar{T}_2)$-bimodules that
\[ \lsub{\bar{\varphi}}\cO\bar{T}_2 \mid \cO\bar{T}_2 \otimes_{\cO{T_2}} i_2\cO{G_2}i_2 \otimes_{\cO{T_2}} \cO\bar{T}_2. \]
Thus $\lsub{\bar{\varphi}}\cO\bar{T}_2$ is isomorphic to an indecomposable direct summand of
\[ \cO\bar{T}_2 \otimes_{\cO{T_2}} \Ind_{R^\sharp}^{P_2\times P_1} V^* \otimes_{\cO{P_1}} i_1\cO{G_1}i_1 \otimes_{\cO{P_1}} \Ind_R^{P_1\times P_2} V \otimes_{\cO{T_2}} \cO\bar{T}_2 \]
by (\ref{setting-SEMT}.\ref{item-iOGi-N*0-N0}) and (\ref{setting-SEMT}.\ref{item-N*0-N0-mid}).
Then by the bimodule structure of source algebras in \cite[Theorem 8.7.1]{Linck18b}, $\lsub{\bar{\varphi}}\cO\bar{T}_2$ is isomorphic to an indecomposable direct summand of
\[ \cO\bar{T}_2 \otimes_{\cO{T_2}} \Ind_{R^\sharp}^{P_2\times P_1} V^* \otimes_{\cO{T_1}}  \lsub{\psi}\Ind_R^{P_1\times P_2} V \otimes_{\cO{T_2}} \cO\bar{T}_2, \]
where $T_1$ is a subgroup of $P_1$ and $\psi\colon T_1 \to P_1$ is in $\cF_1$.

It follows from the structure of $R$ in Proposition \ref{prop-first}(2) and Mackey formula that
\[ \Res^{P_2 \times P_1}_{T_2 \times P_1} \Ind_{R^\sharp}^{P_2 \times P_1} V^* \cong \Ind_{(T_2 \times P_1) \cap R^\sharp}^{T_2 \times P_1} V^*. \]
Restricting the above module further to $T_2 \times T_1$ and using Mackey formula, we have that $\Res^{P_2 \times P_1}_{T_2 \times T_1} \Ind_{R^\sharp}^{P_2 \times P_1} V^*$ is a direct sum of $\cO(T_2\times T_1)$-modules of the form
\[ \Ind_{(T_2 \times T_1) \cap \lsup{(1,s_1)}R^\sharp}^{T_2 \times T_1} U,  \]
where $s_1 \in P_1$ and $U = \lsup{(1,s_1)}V^*$.
Similarly, $\Res^{P_1 \times P_2}_{\psi(T_1) \times T_2} \Ind_R^{P_1 \times P_2} V$ is a direct sum of $\cO(T_1\times T_2)$-modules of the form
\[ \Ind_{[\psi(T_1) \times T_2] \cap \lsup{(s'_1,1)}R}^{\psi(T_1) \times T_2} W_0,  \]
where $s'_1 \in P_1$ and $W_0 = \lsup{(s'_1,1)}V$.
Note that
\[ \lsub{\psi}\Ind_{[\psi(T_1) \times T_2] \cap \lsup{(s'_1,1)}R}^{\psi(T_1) \times T_2} W_0 \cong \Ind_{(\psi^{-1}\times1)\left([\psi(T_1) \times T_2] \cap \lsup{(s'_1,1)}R\right)}^{T_1 \times T_2} W, \]
where $W = \lsub{\psi\times1}W_0$.
Set $T=(T_2 \times T_1) \cap \lsup{(1,s_1)}R^\sharp$ and $T'=(\psi^{-1}\times1)\left([\psi(T_1) \times T_2] \cap \lsup{(s'_1,1)}R\right)$, then
\[ T = \Set{ (t_2,t_1) \in T_2 \times T_1 ~\middle|~ \bar{t}_2 = \left(\theta\circ c_{\bar{s}_1^{-1}}\right)(\bar{t}_1) }, \]
where $c_{\bar{s}_1^{-1}}$ is the inner automorphism of $\bar{P}_1$ induced by $\bar{s}_1^{-1}$, and
\[ T' = \Set{ (t_1,t_2) \in T_1 \times T_2 ~\middle|~ \bar{t}_2 = \left(\theta\circ c_{\bar{s'_1}^{-1}}\circ\bar{\psi}\right)(\bar{t}_1) }, \]
where $c_{\bar{s'_1}^{-1}}$ is the inner automorphism of $\bar{P}_1$ induced by $\bar{s'_1}^{-1}$ and $\bar{\psi}\colon \bar{T}_1 \to \bar{P}_1$ induced by $\psi$ is in $\bar{\cF}_1$ with $\bar{T}_1 = T_1R_1/R_1$.
 
By Lemma \ref{lem-res-ind-quotients}\,(1)(3), we have as $(\cO\bar{T}_2,\cO T_1)$-bimodules that
\begin{align*}
& \cO\bar{T}_2 \otimes_{\cO{T_2}} \Ind_{T}^{T_2\times T_1} U \cong \cO[(T_2\times T_1)/(R_2\times1)] \otimes_{\cO{(T_2\times T_1)}} \Ind_{T}^{T_2\times T_1} U \\
\cong\ & \Ind_{T/(R_2\times1)}^{\bar{T}\times T_1} \left(\cO[T/(R_2\times1)] \otimes_{\cO T} U\right) \cong \Ind_{T/(R_2\times1)}^{\bar{T}_2\times T_1} (\cO \otimes_{\cO R_2} U)
\end{align*}
and similarly that
\[ \Ind_{T'}^{T_1\times T_2} W \otimes_{\cO{T_2}} \cO\bar{T}_2 \cong \Ind_{T'/(1\times R_2)}^{T_1\times \bar{T}_2} (W \otimes_{\cO R_2} \cO). \]
Then it follows that $\lsub{\bar{\varphi}}\cO\bar{T}_2$ is isomorphic to an indecomposable direct summand of
\begin{align*}
& \cO\bar{T}_2 \otimes_{\cO{T_2}} \Ind_{T}^{T_2\times T_1} U \otimes_{\cO{T_1}} \Ind_{T'}^{T_1\times T_2} W \otimes_{\cO{T_2}} \cO\bar{T}_2 \\
\cong\ & \Ind_{T/(R_2\times1)}^{\bar{T}_2\times T_1} U_1 \otimes_{\cO{T_1}}  \Ind_{T'/(1\times R_2)}^{T_1\times \bar{T}_2} W_1,
\end{align*}
where $U_1 = \cO \otimes_{\cO R_2} U$ and $W_1 = W \otimes_{\cO R_2} \cO$.
Thus by Bouc's formula \ref{Bouc}, $\lsub{\bar{\varphi}}\cO\bar{T}_2$ is isomorphic to an indecomposable direct summand of
\[ \Ind_{[T/(R_2\times1)]*\lsup{(r,1)}[T'/(1\times R_2)]}^{\bar{T}_2\times\bar{T}_2} L, \]
where $r \in T_1$, $L$ is an $\cO[T/(R_2\times1)*\lsup{(r,1)}(T'/(1\times R_2))]$-module, and
\begin{align*}
& \bar{T}_0:= [T/(R_2\times1)]*\lsup{(r,1)}[T'/(1\times R_2)] \\
=\ & \Set{ (\bar{t}_2,\bar{t'_2}) \in \bar{T}_2 \times \bar{T}_2 ~\middle|~ \begin{array}{c} \exists\, t_1\in T_1,\ \bar{t}_2 = \left(\theta\circ c_{\bar{s}_1^{-1}}\right)(\bar{t}_1)\\ \bar{t'_2} = \left(\theta\circ c_{\bar{s'_1}^{-1}}\circ\bar{\psi}\right)(\bar{t_1^{r}}) \end{array} }
\end{align*}
by the definition of the operation $*$ in Theorem \ref{Bouc}.
Note that although $t_1$ in the above equation is not unique, the image $\bar{t}_1$ in the quotient group $\bar{T}_1$ is unique.
Thus $\bar{T}_0 \cap (\bar{T}_2\times1)$ and $\bar{T}_0\cap(1\times\bar{T}_2)$ are both trivial.

Since $\lsub{\bar{\varphi}}\cO\bar{T}_2$ has a vertex $\Delta_{\bar{\varphi}}\bar{T}_2$ as $\cO(\bar{T}_2\times\bar{T}_2)$-module, the projections of $\bar{T}_0$ to the first and second components are both surjective.
So $\bar{T}_0 = \Delta_{\theta\circ\bar{\psi'}\circ\theta^{-1}}\bar{T}_2$, where $\bar{\psi'} = c_{\bar{s'_1}^{-1}} \circ \bar{\psi} \circ c_{\bar{r}^{-1}} \circ c_{\bar{s}_1}$ is in $\bar{\cF}_1$.
By Green's indecomposable theorem, there is an indecomposable submodule $L_0$ of $L$ such that
\[ \lsub{\bar{\varphi}}\cO\bar{T}_2 \cong \Ind_{\Delta_{\theta\circ\bar{\psi'}\circ\theta^{-1}}\bar{T}_2}^{\bar{T}_2\times\bar{T}_2} L_0. \]
Since $\lsub{\bar{\varphi}}\cO\bar{T}_2$ is a trivial source module, $L_0$ is in fact the trivial module.
So
\[ \lsub{\bar{\varphi}}\cO\bar{T}_2 \cong \lsub{\theta\circ\bar{\psi'}\circ\theta^{-1}}\cO\bar{T}_2. \]
Then it follows from Lemma \ref{lem-inn-P} that $\bar{\varphi}$ differs from $\theta\circ\bar{\psi'}\circ\theta^{-1}$ by an inner automorphism of $\bar{T}_2$.
Thus we have $\theta^{-1}\circ\bar{\varphi}\circ\theta \in \bar{\cF}_1$.

So by Proposition \ref{thm-quot-fusion-source-alg}\,(4), $\bar{\cF}_2$ can be identified with a fusion subsystem of $\bar{\cF}_1$  via the isomorphism $\theta^{-1}$.
Symmetric arguments shows that $\bar{\cF}_1$ can be identified with a fusion subsystem of $\bar{\cF}_2$ via the isomorphism $\theta$.
Then the assertion of Theorem \ref{mainthm-1}\,(6) follows and the proof completes.
\end{proof}
\end{thmwot}

\section{Proof of Theorem \ref{mainthm-2}}\label{sect:sep-equiv}

In this section, let $\cO$ be as in Convention \ref{conv-O} and assume furthermore that \emph{$\cO$ is of characteristic $0$}; denote by $K$ the field of fractions of $\cO$ and set $J(\cO)=(\pi)$.
We begin with a lemma from \cite{Linck18a}.

\begin{lem}[{\cite[Proposition 4.13.17]{Linck18a}}]\label{lem-sep-div-d}
Let $A,B$ be two $\cO$-algebras and $B$ separably divides $A$.
If $d$ is a positive integer such that $\pi^d$ annihilates $\UHom_A(U,U')$ for any $A$-module $U,U'$, then $\pi^d$ annihilates $\UHom_B(V,V')$ for any $A$-module $V,V'$.
\end{lem}

\begin{proof}
Note that in \cite[Proposition 4.13.17]{Linck18a}, $A,B$ are assumed to be separably equivalent, but the proof there use only the assumption that $B$ separably divides $A$.
So the proof of \cite[Proposition 4.13.17]{Linck18a} applies except some typos which can be corrected easily.
\end{proof}

\begin{thmwot}\label{d_A-sep}
Assume $A$ is an $\cO$-algebra such that $K\otimes_{\cO}A$ is semisimple.
Denote by $d_A$ the smallest positive integer such that $\pi^{d_A}$ annihilates $\UHom_A(U,V)$ for any $A$-module $U,V$ (such $d_A$ exists by \cite[Proposition 4.13.16]{Linck18a}).
Since $p=u\pi^k$ for some $u\in\cO^\times$ and positive integer $k$, there is a smallest power of $p$ annihilating $\UHom_A(U,V)$ for any $A$-module $U,V$.

Assume $A,B$ are two $\cO$-algebras such that $K\otimes_{\cO}A$ and $K\otimes_{\cO}B$ are both semisimple.
By Lemma \ref{lem-sep-div-d}, if $B$ separably divides $A$, then $d_B \leq d_A$.
Thus if $A,B$ are separably equivalent, then $d_A=d_B$.
\hfill $\lrcorner$
\end{thmwot}

\begin{thmwot}\label{proof of mainthm-2}
\begin{proof}[Proof of Theorem \ref{mainthm-2}]
(1.1) just follows from the transitivity of separable equivalences and \cite[Corollary 6.2.4]{Linck18b}.
(1.2) follows from (1.1), Lemma \ref{lem-sep-div-d}, \ref{d_A-sep} and \cite[Corollary 4.13.21]{Linck18a} or follows from Lemma \ref{lem-sep-div-d}, \ref{d_A-sep} and \cite[Corollary 6.2.10]{Linck18b}.

(2.1) By (\ref{setting-SEMT}.\ref{item-N*0-N0-mid}) and (\ref{setting-SEMT}.\ref{item-OP-N*0-N0}), we have that
\[ \cO{P}_2 \mid \Ind_{R^\sharp}^{P_2\times P_1} V^* \otimes_{\cO{P_1}} i_1\cO{G_1}i_1 \otimes_{\cO{P_1}} \Ind_R^{P_1\times P_2} V. \]
Then by \cite[Theorem 8.7.1]{Linck18b}, there is $\varphi_1\colon Q_1 \to P_1 \in\cF_1$ such that
\[ \cO{P}_2 \mid \Ind_{R^\sharp}^{P_2\times P_1} V^* \otimes_{\cO{Q_1}}  \lsub{\varphi_1}\Ind_R^{P_1\times P_2} V. \]
Applying \cite[Theorem 2.7.2]{Linck18a} to the following biadjoint functors:
\begin{align*}
G &= \Ind_{R^\sharp}^{P_2\times P_1} V^* \otimes_{\cO{Q}_1} \colon\ \cO{Q}_1\otimes_{\cO}(\cO{P}_2)^{\op}\!-\!\Mod \to \cO{P}_2\otimes_{\cO}(\cO{P}_2)^{\op}\!-\!\Mod, \\
F &= \Ind_R^{P_1\times P_2} V \otimes_{\cO{P}_2} \colon\ \cO{P}_2\otimes_{\cO}(\cO{P}_2)^{\op}\!-\!\Mod \to \cO{Q}_1\otimes_{\cO}(\cO{P}_2)^{\op}\!-\!\Mod,
\end{align*}
we have that
\[ \cO{P}_2 \mid \Ind_{R^\sharp}^{P_2\times P_1} V^* \otimes_{\cO{Q_1}}  \Ind_R^{P_1\times P_2} V \otimes_{\cO{P}_2} \cO{P}_2 \cong \Ind_{R^\sharp}^{P_2\times P_1} V^* \otimes_{\cO{Q_1}}  \Ind_R^{P_1\times P_2} V. \tag{$\ast$}\]
Thus it follows from Lemma \ref{lem-sep-div-d} and \cite[Corollary 4.13.21]{Linck18a} that $|P_2|\leq|Q_1|\leq|P_1|$.
Symmetric arguments shows that $|P_1|\leq|P_2|$.
So $Q_1=P_1$ and ($\ast$) becomes
\[ \cO{P}_2 \mid \Ind_{R^\sharp}^{P_2\times P_1} V^* \otimes_{\cO{P_1}}  \Ind_R^{P_1\times P_2} V. \tag{$\ast\ast$}\]
Similarly, we have
\[ \cO{P}_1 \mid \Ind_R^{P_1\times P_2} V \otimes_{\cO{P_2}} \Ind_{R^\sharp}^{P_2\times P_1} V^*. \]

(2.2) By Mackey's formula and the surjection of $\pi_1$ and $\pi_2$ in Proposition \ref{prop-first}(2), we have that
\[ \Res^{P_2\times P_1}_{R_2\times P_1} \Ind_{R^\sharp}^{P_2\times P_1} V^* \cong \Ind_{R_2\times R_1}^{R_2\times P_1} V^*,\quad \Res^{P_1\times P_2}_{P_1\times R_2} \Ind_R^{P_1\times P_2} V \cong \Ind_{R_1\times R_2}^{P_1\times R_2} V. \]
Then it follows from ($\ast\ast$) that
\[ \cO{R}_2 \mid V^* \otimes_{\cO{R}_1} \cO{P}_1 \otimes_{\cO{R}_1} V. \]
Then we can argument as in (2.1) to show that
\[ \cO{R}_2 \mid V^* \otimes_{\cO{R}_1} V,\quad \cO{R}_1 \mid V \otimes_{\cO{R}_2} V^*. \qedhere\]
\end{proof}
\end{thmwot}

\begin{rem}
For the field $k$, the above arguments do not apply.
\begin{compactenum}[(1)]
\item
When $kP_1$ and $kP_2$ are separably equivalent, S.F. Peacock shows that $P_1$ and $P_2$ have the same $p$-rank (see \cite[p.225, Theorem 1]{Peac17} and the paragraph after it); but it is not known whether $|P_1|=|P_2|$ holds even for cyclic $p$-groups.
\item
Assume $kG_1b_1$ is stably equivalent of Morita type to $kG_2b_2$ and keep the notation in Theorem \ref{mainthm-1}.
Then as in \ref{proof of mainthm-2}, we can show that
\[ kP_2 \mid \Ind_{R^\sharp}^{P_2\times P_1} V^* \otimes_{kQ_1}  \Ind_R^{P_1\times P_2} V. \]
It follows from Mackey's formula that
\[ \Res^{P_2\times P_1}_{P_2\times Q_1} \Ind_{R^\sharp}^{P_2\times P_1} V^* \cong \Ind_{R^\sharp\cap(P_2\times Q_1)}^{P_2\times Q_1} V^*,\quad \Res^{P_1\times P_2}_{Q_1\times P_2} \Ind_R^{P_1\times P_2} V \cong \Ind_{(Q_1\times P_2)\cap R}^{Q_1\times P_2} V. \]
Then by Bouc's formula, we have that
\[ \Ind_{R^\sharp}^{P_2\times P_1} V^* \otimes_{kQ_1} \Ind_R^{P_1\times P_2} V \cong \Ind_{[R^\sharp\cap(P_2\times Q_1)]*[(Q_1\times P_2)\cap R]}^{P_2\times P_2} \left[V^*\otimes_{k(Q_1\cap R_1)} V\right]. \]
Since as a $(kP_2,kP_2)$-bimodule $kP_2$ has $\Delta{P_2}$ as a vertex, the projections of $[R^\sharp\cap(P_2\times Q_1)]*[(Q_1\times P_2)\cap R]$ to the first and second components must be surjective.
By the definition of the operator ``$*$'' in Bouc's formula, we have that $P_1=Q_1R_1$.
But we do not known whether it can be shown that $Q_1=P_1$. \hfill $\lrcorner$
\end{compactenum}
\end{rem}

\section{Examples and remarks}\label{exmp-rem}

\begin{exmp}\label{exmp-1}
At least for $k$, there are examples in which stable equivalences of Morita type can have non-trivial $R_1$ and $R_2$ as in Theorem \ref{mainthm-1}.

In fact, there are counterexamples for the so-called \emph{modular isomorphism problem} (MIP) for $p$-groups:
\begin{equation}
\begin{array}{c}
\text{if $P_1,P_2$ are two $p$-groups such that $kP_1 \cong kP_2$ as $k$-algebras,}\\
\text{then does it hold that $P_1 \cong P_2$?}
\end{array}
\end{equation}
The first counterexample is given by Garc\'ia-Lucas, Margolis and del R\'io \cite{G-LMargRio22}.

Note that $kP_1,kP_2$ are themselves blocks, and $kP_1 \cong kP_2$ implies obviously Morita equivalence and thus stable equivalences of Morita type between $kP_1$ and $kP_2$.
Thus subgroups $R_1,R_2$ as in Theorem \ref{mainthm-1} are non-trivial in the counterexample given by Garc\'ia-Lucas, Margolis and del R\'io \cite{G-LMargRio22}, since otherwise $P_1\cong P_2$ would hold.

Note that the counterexample of Garc\'ia-Lucas, Margolis and del R\'io \cite{G-LMargRio22} is also counterexample for the \emph{blockwise modular isomorphism problem} (see \emph{e.g.} Navarro and Sambale \cite[Question 1.1]{NavSam18} for the statement). \hfill $\lrcorner$
\end{exmp}

\begin{exmp}\label{exmp-2}
By a result of Puig \cite[7.4]{Puig99}, the consition $R_1=1=R_2$ in part (7) is equivalent to that $M$ has endopermutation sources.
Thus part (7) is just a result of Puig \cite[7.6]{Puig99};
see \cite[Theorem 9.11.2]{Linck18b} for the statement and a proof:
in fact, when the isomorphism of defect groups is established from the assumption of endopermutation sources in \cite[Theorem 9.11.2]{Linck18b}, the remaining proof for the isomorphism of fusion systems in \cite[Theorem 9.11.2]{Linck18b} does not make further use of the assumption of endopermutation sources.

Note that the counterexample of Garc\'ia-Lucas, Margolis and del R\'io \cite{G-LMargRio22} is also a counterexample for the following question over $k$:
\begin{equation}\label{equ-basic-Morita}
\begin{array}{c}
\text{do all Morita equivalences between blocks of finite groups}\\
\text{have endopermutation sources?}
\end{array}
\end{equation}
(It is open for $\cO$ when $\cO$ is of characteristic $0$.)
Furthermore, there are even Morita equivalence between blocks of finite groups over $k$ with isomorphic defect groups which does not have endopermutation sources.
\emph{Linckelmann gives such a counterexample, which I learn from Dr. X. Huang} and the groups in this example are much smaller than those provided by Garc\'ia-Lucas, Margolis and del R\'io \cite{G-LMargRio22}.
Note first that stable equivalences of Morita type with endopermutation sources perserve vertices of indecomposable modules.
Thus to give an example of Morita equivalence whose sources are not endopermutation modules, it suffices to construct a Morita equivalence which does not preserve vertices of indecomposable modules.
Then \emph{the example given by Linckelmann is as follows}:
let $G$ be the Klein four group, $p=2$ and $|k|\geq4$;
the indecomposable $2$-dimensional $kG$-modules are parametrized by a projective line (\cite[Theorem 7.2.1(iii)]{Linck18b});
since indecomposable $2$-dimensional $kG$-modules with vertices of order $2$ must be of trivial source, there are only three such indecomposable modules corresponding to the three subgroups of $G$ of order $2$;
the paragraph after \cite[Theorem 7.2.1]{Linck18b} shows that the automorphism group of $kG$ permutes transitively the indecomposable $2$-dimensional $kG$-modules;
thus certain Morita equivalence induced by a suitable automorphism of $kG$ should send an indecomposable $2$-dimensional $kG$-module with vertex of order $2$ to an indecomposable $2$-dimensional $kG$-module with vertex $G$.
\end{exmp}

\begin{thmwot}\label{Ques-R1R2}
Finally, we propose some question related to Theorem \ref{mainthm-1}.
\begin{compactenum}[(\ref{Ques-R1R2}.1)]
\item
Does it hold that $R_1$ and $R_2$ are indeed strongly closed in $\cF_1$ and $\cF_2$ resp.?
If this assertion does hold, then by \cite[Proposition 5.19]{Craven11}, part (6) of Theorem \ref{mainthm-1} can be reformulated as
\[ N_{\cF_1}(R_1)/R_1 \cong N_{\cF_2}(R_2)/R_2. \]
So is there any local meaning of the above isomorphism?
\item
The minimal possibility for $R_1,R_2$ is of course $R_1=1=R_2$, and by a result of Puig \cite[7.4]{Puig99}, this is equivalent to that the stable equivalence of Morita type has endopermutation sources.
So a natural question is: what is the maximal possibility of $R_1$ and $R_2$?
In particular, can the situation in which $R_1=P_1$ and $R_2=P_2$ happen?
\item
Fixing two blocks $b_1$ and $b_2$, among all stable equivalences of Morita type (Morita equivalences) between them, the choice of $R$ is unique up to conjugacy or not?
And if this uniqueness does not hold, what can we say about the set of all possible $R$?
\end{compactenum}
Note that if the problem (\ref{equ-basic-Morita}) holds for $\cO$ when $\cO$ is of characteristic $0$, then in this case we always have $R_1=1=R_2$ and the above question become trivial.
But by the counterexamples in \ref{exmp-1} and \ref{exmp-2}, the above questions make sense at least for $k$.
\end{thmwot}

\section*{Acknowledgement}
The author thank Prof. Z. Feng for his invitation for a stay at Shenzhen International Center for Mathematics, Southern University of Science and Technology (SUSTech) in the summer of 2025, where this work was completed.
The author is extremely grateful to Prof. M. Linckelmann for helpful comments, and to Prof. X. Huang for fruitful disscusions.



\begin{thebibliography}{99}\setlength{\itemsep}{-2pt}
\small

\bibitem{AschKesOliv11}
M. Aschbacher, R. Kessar, and B. Oliver, Fusion systems in algebra and topology, London Math. Soc. Lecture Notes Series 391, Cambridge University Press, 2011.

\bibitem{BerErd11}
P.A. Bergh, K. Erdmann, The representation dimension of Hecke algebras and sym-metric groups, Adv. Math. 228 (4) (2011) 2503--2521.

\bibitem{Bouc10}
S. Bouc, Bisets as categories and tensor product of induced bimodules. Appl. Categor. Struct. 18 (2010), 517--521.

\bibitem{Brou94}
M. Brou\'e, Equivalences of blocks of group algebras, in: Finite dimensional algebras and related topics, Kluwer (1994), 1--26.

\bibitem{Craven11}
D.A. Craven, The theory of fusion Systems, Cambridge Studies in Advanced Mathematics 131, Cambridge University Press, 2011.

\bibitem{G-LMargRio22}
D. Garc\'ia-Lucas, L. Margolis, \'A. del R\'io, Non-isomorphic $2$-groups with isomorphic modular
group algebras, J. Reine Angew. Math. 783 (2022), 269--274.

\bibitem{Kad95}
L. Kadison, On split, separable subalgebras with counitality condition, Hokkaido Math. J. 24 (3) (1995). 527--549.

\bibitem{Lich25}
C. Li, On stable equivalences of Morita type and nilpotent blocks, J. Algebra 665 (2025), 243--252.

\bibitem{LiTian24}
C. Li, Y. Tian, Non-injective inductions and restrictions of modules over finite groups, Czechoslovak Math. J. 75(3) (2025), 933--942.

\bibitem{Linck07}
M. Linckelmann, Introduction to fusion systems, in: Group representation theory EPFL Press, Lausanne (2007), 79--113.

\bibitem{Linck11}
M. Linckelmann, Finite generation of Hochschild cohomology of Hecke algebras of finite classical type in characteristic zero, Bull. Lond. Math. Soc. 43(5) (2011) 871--885.

\bibitem{Linck18a}
M. Linckelmann, The block theory of finite group algebras, Vol. I, London Mathematical Society Student Texts, 91. Cambridge University Press, Cambridge, 2018.

\bibitem{Linck18b}
M. Linckelmann, The block theory of finite group algebras, Vol. II, London Mathematical Society Student Texts, 92. Cambridge University Press, Cambridge, 2018.

\bibitem{NT}
H. Nagao, Y. Tsushima, Representations of finite groups. Academic Press, INC. 1987.

\bibitem{NavSam18}
G. Navarro, B. Sambale, On the blockwise modular isomorphism problem, Manuscripta Math. 157 (2018), no.1--2, 263--278.

\bibitem{Peac17}
S.F. Peacock, Separable equivalence, complexity and representation type, J. Algebra 490 (2017), 219--240.

\bibitem{Puig86}
L. Puig, Local fusions in block source algebras, J. Algebra 104 (1986), no. 2, 358--369.

\bibitem{Puig99}
L. Puig, On the local structure of Morita and Rickard equivalences between Brauer blocks, Progress in Mathematics, 178, Birkh\"auser Verlag, Basel, 1999.

\bibitem{Puig06}
L. Puig, Frobenius categories, J. Algebra 303:1 (2006), 309--357.

\bibitem{Rick91}
J. Rickard, Derived equivalences as derived functors, J. London Math. Soc. 43 (1991), 37--48.

\bibitem{Suzuki82}
M. Suzuki, Group Theory I, Springer-Verlag, 1982.

\bibitem{Xi08}
C. Xi, Stable equivalences of adjoint type, Forum Mathematicum 20 (2008), 81--97.

\bibitem{Zimmer14}
A. Zimmermann, Representation theory, A homological algebra point of view, Algebra and Applications 19, Springer, 2014.

\end{thebibliography}
\end{document}